\documentclass[runningheads]{llncs}
\usepackage{cite}
\usepackage{enumitem}
\usepackage{amsfonts, amsmath, amssymb}
\usepackage[nameinlink]{cleveref}  % For generating a new enumberated environment with reference links
\usepackage{cancel, ulem}
\usepackage{array}

\usepackage{tikz, color, soul}
\usepackage{pgfplots}
\usetikzlibrary{plotmarks}
\usepackage{mathtools}
\usepackage[ruled,linesnumbered]{algorithm2e}
\usetikzlibrary{shapes.geometric,arrows,positioning}
\tikzstyle{decision} = [diamond, draw, fill=blue!20, 
    text width=4.5em, text badly centered, node distance=4cm, inner sep=0pt]
\tikzstyle{block} = [rectangle, draw, fill=blue!20,  text centered, rounded corners, minimum height=4em]
\tikzstyle{line} = [draw, -latex']
\tikzstyle{cloud} = [draw, ellipse,fill=red!20, node distance=6.6cm,
    minimum height=1em]

\newtheorem{Theorem}{Theorem}
\newtheorem{Definition}{Definition}

\newtheorem{Corollary}{Corollary}
\newtheorem{Proposition}{Proposition}
\newtheorem{Remark}{Remark}

\newtheorem{Claim}{Claim}
\newtheorem{Example}{Example}
\newcounter{casenum}
\newenvironment{caseof}
  {\setcounter{casenum}{0}}
  {\par\addvspace{.5\baselineskip}}

\renewcommand*{\thecasenum}{\arabic{casenum}}

\newcommand{\tcase}[2]{%
  \par\addvspace{.5\baselineskip}%
  \noindent \refstepcounter{casenum}\textbf{Case \thecasenum:}~#1\\*
  #2\ifhmode\unskip\fi
}

\crefname{casenum}{\protect\textbf{case}}{\protect\textbf{cases}}
\Crefname{casenum}{\protect\textbf{Case}}{\protect\textbf{Cases}}
\creflabelformat{casenum}{#2\textbf{(#1)}#3}

\providecommand*{\Xmath}[1]{\ensuremath{#1}\xspace}

\newcommand{\N}{{\mathbb N}}

\newcommand{\D}{\Xmath{\mathcal D}}
\def\E{{\mathcal E}}
\def\N{{\mathcal N}}

\def\T{{\mathcal T}}

\def\NN{{\mathbb N}}

\def\ZZ{{\mathbb Z}}

\usepackage{xcolor}
\newif\ifcomment
\commenttrue
\renewcommand{\epsilon}{\varepsilon}

\newcommand{\A}{\mathcal{A}}

\usepackage{mathtools}
\newcommand{\defeq}{\vcentcolon=}

\newcommand{\bN}{{\mathbb N}}
\newcommand{\pn}{P_N}  % Path with N nodes
\newcommand{\ton}{\lbrack N \rbrack}  % {1, 2, ..., N}
\newcommand{\tton}{\lceil n \rceil}  % {0, 1, ..., N}

\def\qed{\hfill $\square$}
\pgfplotsset{compat=1.15}
\newcounter{modulo}

\begin{document}
        \title{Non-Adaptive and Adaptive Two-Sided Search with Fast Objects\thanks{Part of the results were presented at ISIT 2021 in \cite{LD21}.}}
        \author{Alexey Lebedev\inst{1} \and Christian Deppe\inst{2}}
        \institute{ Kharkevich Institute for Information Transmission Problems\\
             Russian Academy of Sciences, Moscow, Russia \\
             \email{al\_lebed95@mail.ru}
             \and
             Institute for Communications Engineering\\
                          Technical University of Munich, Munich, Germany\\
        \email{christian.deppe@tum.de} 
        %\email{lebedev37@mail.ru al_lebed95@mail.ru}
        }
        
        \maketitle

\begin{abstract}
In 1946, Koopman introduced a two-sided search model. In this model, a searched object is active and can move, at most, one step after each test. We analyze the model of a combinatorial two-sided search by allowing more moves of the searched object after each test. 
We give strategies and show that they are optimal. We consider adaptive and non-adaptive strategies.
We show the surprising result that with the combinatorial two-sided search on a path graph, the optimal non-adaptive search needs the same number of tests as the corresponding adaptive strategy does.
The strategy obtained can also be used as a encoding strategy to sent the position of a moving element through a transmission channel.
\end{abstract}

\section{Introduction}

Koopman worked on search theory during World War II. His models and results were finally published in his book \cite{K46} in 1946 after the end of the war. In this work he introduced the term two-sided search. 
The name two-sided search introduced by Koopman for this problem is meant to express that unlike in traditional search, the searched object can "react" by changing its position. The name may be a bit unfortunate, since the search itself is only performed from one side. It should be expressed that both sides (seeker and searched object) are active.
Koopman considered probabilistic search models. 

%A combinatorial model was first introduced in the work 
%\cite{ACDL13}. 

Ahlswede had the idea of investigating a combinatorial model of searching for moving elements. He raised this question during a discussion in the "Search Methodologies II" workshop at the ZiF in 2010, and the first results were published in \cite{ACDL13}. For detailed results and possible applications we refer the reader to \cite{ACDL18}. 
Previously, similar models have been considered.
A related model was viewed in the form of a hidden object game in \cite{BK16}. 
In \cite{T86}, a cops and robber game was introduced. This model was also studied in \cite{BGGK09} and \cite{DDTY15}. In \cite{BW13} and in \cite{AFGP16}, a hunters and rabbit game was introduced. Further, in \cite{BL21}, the authors considered searching for an intruder on graphs. While all of these models have some similarities to our model, the differences lead to the need for a different analysis of our problem. 

A combinatorial search problem is analyzed in a discrete space and consists of finding a set of items satisfying specified requirements.
Nowadays, combinatorial search  involves an extensive number of challenging  optimization problems which come directly from practical applications.
The fundamentals of combinatorial search can be found in 
\cite{A87}, \cite{Aig88} and \cite{K73}.

The two-sided search strategies considered so far are for the case where the searched object can move, at most, one step further after each test. In the application considered by Koopman (searching for a boat), it was possible that the object changed its speed significantly. This is also relevant for other applications. Therefore we analyze the case where the searched object takes not just one step, but up to $k$ steps after each successful test. We assume that $k$ is known to the searcher.

So far, only adaptive two-sided search strategies have been considered. In adaptive search strategies, each test depends on the results of the previous tests. However, in practice, it is not always possible to get the test results before the next test. In this case, we need non-adaptive search strategies where the tests do not depend on the previous results.

In similar purely combinatorial search problems, adaptive search methods often require significantly fewer tests than in the cases of non-adaptive search methods (see \cite{C13}). This is different in probabilistic memoryless search problems (see \cite{A12, B13}). Here there are models in which the number of non-adaptive tests and adaptive tests hardly differ.
We give an example of a purely combinatorial non-adaptive two-sided search strategy that requires as many tests as the adaptive strategy and do a worst-case analysis of our model. This means that the object moves in such a way that we need as many tests as possible. The object follows an adversarial strategy, so to speak.
% The problem can also be seen as a generalization of group testing, where the objects being searched are moving. We consider the special case of one searched object.
% As in many cases, there is an equivalent coding problem for this search problem. Here a sender wants to transmit the position of an object to the receiver as accurately as possible as the object moves within each time step.
The paper is organized in following manner: In Section~2 we describe the general model for an adaptive and non-adaptive search. 
In Section~3 we solve the problem of finding an optimal adaptive strategy for
searching for faster objects on a cycle graph and a path graph.
In Section~4 we present a non-adaptive strategy for the search on a path graph.
In Section~5 we explain the equivalent coding problem in more detail.

\section{Model and Definitions}
\label{model}
We follow the same notations as given in \cite{ACDL13}. Therefore the model is first defined on a general graph, and afterward special cases are treated.
We let $\N=\{1,2,\dots,N\}$ denote the search positions of the searched object. The search space is modelled by the graph $G=(\N,\E)$. A searched object, also called a target, occupies one of the vertices in $\N$, unknown to the searcher. The searcher is able to detect the presence of the target at any subset of $\N$, i.e., for any $\T_i\subseteq \N$, which is called the $i$th test set. The searcher can learn whether the target is located at $\T_i$ or not. He receives 
$y_i=1$ as a test result if the target is located in $\T_i$, otherwise he receives $y_i=0$. 
 After each test,
 the target can move at most $k$ times to an adjacent vertex, or stay in the same place. 
For ease of description, we assume that each vertex in our graph $G$ has a loop. Thus
we may formally assume that in each time unit
the target moves $k$ times to an adjacent vertex. 
The goal is to find the object with accuracy $s$. This means that the searcher can specify a subset with cardinality $s$ in which the object will be located in the next step.
One cannot achieve accuracy $s=1$ in general
because the object can still move after the last test.
This is a generalization of the model with $k=1$, as was considered in  \cite{ACDL13}.

We assume that the search starts when the target is at vertex $d_1$. The positions of the target before each test are denoted by $d^{n+1} \defeq (d_1,\dots,d_{n+1}) \in \N^{n+1}$, where there is a walk of length at most $k$ between $d_i$ and $d_{i+1}$ for all $i\in\{1,\cdots,n\}$. Thus the vector $d^{n+1}$ indicates that the target occupies the vertex $d_j$ at time $j$. Recall that in standard search models (without movement) $d_i=d_j \text{ for } 1 \leq i,j \leq n+1$.
%\msa{this sentence can be removed, it is already mentioned in the 2nd sentence of current paragraph!}

Next we give a formal description of our search model.
Let $d_1\in \N$ be the initial unknown position of the target and let
 $(\T_1, \T_2, \dots, \T_n)$ be a sequence of test sets $\T_i\subseteq \N$ (tests for short), which are performed successively in time.
 
  We also analyze a non-adaptive version of our problem. Even in the non-adaptive setting, the target can move up to $k$ steps between any two tests.
 In this case the tests are not able to depend on 
previous test results. The non-adaptive tests are still executed sequentially. It is not possible to execute them simultaneous. One reason is that  the time units are so short that one cannot evaluate their results before deciding on the next test set.
Therefore, in a non-adaptive search strategy, each test is independent of the results of the other tests.

 In the case of an adaptive search strategy, the $i$th test depends on the $i-1$ previous test results. To be precise, in an adaptive case, each test set $\T_i$ is a value of a function depending on the sequence $y^{i-1}=(y_1,\dots,y_{i-1})$.
We also characterize the adaptive strategy by the following notation
$\left(\T_1, \T_2(y^1), \dots, \T_n(y^{n-1})\right)$.
Let $(d_1,\ldots, d_{n+1})$ specify the positions occupied by the target during the search.
For each test $\T_i$, the test function is defined as $f_{\T_i}(d_i)=y_i$.
%%%

We call $(\T_1, \T_2(y^1), \dots, \T_n(y^{n-1}))$ a strategy of length $n$.
  %\msa{its not clear that what is the exact definition of an optimal strategy here! do you mean by strategy the same as sequence of test sets?}
  A non-adaptive strategy $S(n,N)$ can be represented by an $n\times N$ test matrix $A_{n,N}=\{a_{ij}\}$,
 where
 \begin{align*}
     a_{ij} =
     \begin{cases}
    1 & j\in \T_i
     \\
     0 & \text{otherwise}
     \end{cases}, \text{ for $1 \leq i \leq n$ and $1 \leq j \leq N$.}
 \end{align*}
It follows from the definition that the $i$th row of the matrix $A_{n,N}$ represents the test $\T_i$. The set of possible positions of the target after the $i$-th test is specified by $\D_i$, thus $\D_0=\N$. $P_i^j$ 
is a variable to denote arbitrary paths from $i$ to $j$ in $G=(\N,\E)$. The number of edges on the path $P_i^j$ is referred to as the length of $P_i^j$ and is denoted by $l(P_i^j)$. 
\begin{Definition}
For $\A\subseteq \N$, let $\Gamma_k (\A)$ be the $k$-neighborhood of $\A$ which is defined as follows:
 %%%
  \begin{align*}
      \Gamma_k (\A):=\{j\in\N : for\ some\ i\in \A\ there\ exists\ P_i^j \text{ with } l(P_i^j) \leq k \}.
  \end{align*}
 %%%%
 \end{Definition}
 Observe that $\A\subseteq \Gamma_k(\A)$.

 Further, for the set of possible positions we obtain
\begin{equation*}
\D_i=\left\{
\begin{array}{ll}
\Gamma_k(\T_i\cap \D_{i-1})  &,\ {\rm if\ } f_{\T_i}(d_i)=1\\
\Gamma_k(\D_{i-1}\backslash \T_i)  &,\ {\rm if\ }f_{\T_i}(d_i)=0.
\end{array}\right.
\end{equation*}
We say that the test $\T_i$ reduces $\D_{i-1}$ to $\D_{i}$.

Given a graph $G=(\N,\E)$, for any $s\in\NN$ a strategy with 
$n$ tests is called $(G,s)$-successful
if $|\D_i|\leq s$ for some $1 \leq i \leq n$.
We call $s$ the accuracy of the strategy.
Let
%%%
\begin{align*}
    s^*(G) = \min\{s \in \mathbb{N}~:~\text{ there exists a } (G,s)\text{-successful strategy}\}.
\end{align*}
%%%
Given an integer $s\geq s^*(G)$, we define 
%%%
\begin{align*}
    n(G,s) = \min\{n \in \mathbb{N}~:~\text{there exists a } (G,s)\text{-successful strategy with $n$ tests}\}.
\end{align*}
%%%
The corresponding strategy is then called an optimal $(G,s)$ strategy.

% We call a strategy adaptive
%  if the result of all previous tests $\T_i$ ($1\leq i\leq n-1$), 
%  %\msa{should be $1\leq i\leq n$, right?} $0$
%   defined by $f_{\T_i}(d_i)$, can be used for the next test $\T_{i+1}$. Each non-adaptive strategy is also as an adaptive strategy.

\section{Adaptive Search for Fast Objects}
In this section we consider the case where the object is allowed to move at most $k$ positions after each test. 
In \cite{ACDL13} the case k=1 was solved for circles, paths and trees. For other graphs there are no known solutions. Instead of allowing the object to take $k$ steps in a graph, one could also connect all vertices in the $k$-neighborhood with edges in the graphs for each vertex. Thus, our result implies new results for slowly moving objects ($k=1$) on the k-th powers of paths and cycles. 
%%%%%%%%%%%%%%%%%%%%%%%%%%%%%%%%%%%%%%%%%%%%%%%%%%
\subsection{Cycle}
Let $C_N=(\N,\E)$  be an undirected cycle of length $N$ with a loop in each node. Then the set of vertices and edges are expressed as follows:
%%%
\begin{gather*}
    \N = \{1,\dots,N\},
    \\
    \E = \{(i,i)\}_{i \in \N}~\bigcup~\{N,1\}~\bigcup~\{(i,i+1)\}_{i\in\N\backslash N}.
\end{gather*}
 We denote by $N_c(n,s)$ the maximum $N$, such that there exists an $(C_N,s)$-successful strategy with $n$ tests. 
\begin{Proposition}
\label{s5}
If $k\geq 1$ then the following holds true:
\begin{enumerate}[label=(\roman*)]
\item For $1\leq N \leq 4k$ there does not exist a  $(C_N,s)$--successful   strategy with $s<N$, that is $s^*(C_N)=N$.
\item For $N\geq 4k+1$ there does not exist a  $(C_N,s)$--successful   strategy with $s\leq 4k$, that is $s^*(C_N)\geq 4k+1$.
\end{enumerate}
\end{Proposition}
%%%
\begin{proof}
For the proof, we can assume without loss of generality that the tests consist of  consecutive elements, since the number of possible positions of the searched object is greater for non-consecutive tests of the same cardinality. As in \cite{ACDL18}, for every successful test strategy with non-consecutive tests, one can find one with at least the same number of tests using only consecutive tests.
\begin{enumerate}[label=(\roman*)]
    \item Let $1\leq N\leq 4k$.
    %\msa{can be removed, it is obvious from (i) that what is the range of $N$, or perhaps replacing 'Let' by 'Since'} I think it is OK.
    Then $\D_0=\{1,\dots,N\}$ is the set of possible targets before the first test. For any test $\T$ of a given strategy it holds
\[
|\D_1|=\max\{|\Gamma_k(\T\cap \D_{0})|,
|\Gamma_k(\D_{0}\backslash \T)|\}=|\D_0|.
\]
This is because the test $\T$ or the complementary test $\D_{0}\backslash \T$ contain at least $\lceil \frac N2\rceil$ elements and the $k$-neighborhood of this set is $\{1,\dots,N\}$ again.
Looking at the worst case, the set of possible positions cannot be reduced.

\item We assume that $N\geq 4k+1$.
%\msa{the same argument as above in $(i)$, it can be removed} s.above
Again, $\D_0=\{1,\dots,N\}$ is the set of possible targets before the first test. In addition, the following holds true.
\[
\max\left\{|\Gamma_k(\T_i\cap \D_{i-1})|,
|\Gamma_k(\D_{i-1}\backslash \T_i)|\right\}\geq \left\lceil \frac {|\D_{i-1}|}2 \right\rceil +2k,
\]
because the best possible test halves the number of the possible positions of the object and the target could move by $k$ to each side. By induction on $i$, we have
\[
\left\lceil \frac {|\D_{i-1}|}2 \right\rceil +2k\geq 4k+1.
\]
This completes the proof of Proposition~\ref{s5}.
\end{enumerate}
%\msa{it is not explained that why above inequality is true! is that because of induction or other reason? it would be good to explain why that holds true.}
\qed
\end{proof}

\begin{Theorem}\label{cycle}
For $k\geq 1$, $s\geq 4k$ and $n\geq 0$, the maximum $N$, such that there exists an $(C_N,s)$-successful strategy, is given by
\[
N_c(n,s)= 2^n(s-4k)+4k.
\]
\end{Theorem}
\begin{proof}
Let $k\geq 1$, 
%\msa {this can be added to the theorem statement not here}
$s\geq 4k$ 
%\msa{can be removed} see above
and $N=(s-4k)2^n+4k$. The addition modulo $N$ is denoted by $\oplus$. First we introduce a successful strategy as follows. Note that $\D_0=\N=\{0,1,2,\dots,N-1\}$. We describe the strategy inductively, i.e., let $x\in\N$ be chosen such that
\[
\D_{i-1}=\{x\oplus 1, x\oplus 2, \dots, x\oplus 2^{n-(i-1)}(s-4k)+4k\}.
\]
Then we choose the test
\[
\T_i=\{x\oplus 1, x\oplus 2, \dots, x\oplus 2^{n-i}(s-4k)+2k\}.
\]
Observe that after $i$ tests by induction, we obtain
\[
|\D_i|= 2^{n-i}(s-4k)+4k.
\]
After $n$ tests we have $|\D_n|=s$.

To prove the upper bound, we consider $N=(s-4k)2^n+4k+1$ and consider any strategy with $n$ tests. We will show that this 
strategy is not successful. Let $S$ be a strategy with the
tests $\T_1,\T_2,\dots,\T_n$. Now we consider $\D_1,\dots,\D_n$. Note that there is always a test result such that
\[
|\D_i|=\max\{|\Gamma_k(\T_i\cap \D_{i-1})|,
|\Gamma_k(\D_{i-1}\backslash \T_i)|\}.
\]
Since the searched object can move by $k$ positions after each test, we get
\[
|\D_i|\geq \left\lceil \frac {|\D_{i-1}|}2\right\rceil + 2k.
\]
We have $|\D_0|=(s-4k)2^n+4k+1$ and therefore 
$|\D_i| \geq (s-4k)2^{n-i}+4k+1$ for all $i$. We get $|\D_n|=s+1$, and therefore
$S$ 
%\msa{what is $S$? here is the first time that I see the notation $S$!} see above!!!!!
is not successful.
\qed
\end{proof}

\subsection{Path}

Let
$N \in \bN^+.$ A path $\pn$ is an undirected graph
\begin{align*}
\pn = (\ton, \{(1, 1)\}\cup \bigcup_{2 \leq i \le N}{\{(i-1, i), (i, i) \}}).\label{eq:1}
\end{align*}
We denote by $N_p(n,s)$ the maximum $N$, such that there exists a
$(P_N,s)$-successful strategy with $n$ tests.  % TODO: Define $(P_N,s)$-successful strategy with $n$ tests
%\msa{I think this sentence can be removed, if you mean the result of Propos.~2, then perhaps after the Propos.~2 write a remark or sentence that unlike the cycle ... but before the Propos.~2 and without any proof, it does not make so much sense to bring a result ...} Sentence is changed.
%
\begin{Proposition}\label{pcase}
  For a path $\pn$ with $k \in \bN^+, N \geq 4k+1$ vertices, there does not exist an $(P_N,s)$-successful strategy
  with $s \leq 3k,$ i.e.,
  \begin{align*}
    s^{*}(\pn)\geq 3k+1.
  \end{align*}
\end{Proposition}

% \begin{Proposition}
% Let $k\geq 1$. For $N\geq 4k+1$ there does not exist a $(P_N,s)$-successful strategy with $s\leq 3k$, i.e.,
% \begin{align*}
%     s^*(P_N)\geq 3k+1.
% \end{align*}
% \end{Proposition}
%
\begin{proof}
We can define an adversarial
  strategy $(d_i)_{i \in \tton}$ for the target, i.e., a sequence of test results that guarantees the following claim, and which implies the desired result.
  % Instead of proving the proposition directly, we prove the following claim.
  \begin{Claim}
Let T be a strategy with $n \in \bN^{+}$ tests. Let $N \geq 4k+1.$ For each $i \in \tton$, there exist
    % $j_i \in \D_i,$ s.t. $\D_i \supseteq\{ j_i,j_i+1,\dots,j_i+ 3k\} =\vcentcolon \tilde{\D}_i$
    a sequence $\tilde{\D}_i \vcentcolon= \{j_i,j_i+1,\dots,j_i+ 3k\} \subseteq \D_i.$
    % s.t. there exist a sequence ... \supsetneq
  \end{Claim}
  We prove the claim by induction. \\
  {\bf Induction basis:} Let $j_0= 1$. Then the claim trivially holds for $\D_0$, since
  $\N = \D_0\supseteq \{1,2,\dots,3k+1\}$. \\
  {\bf Induction step:} Assume now that the statement holds for an
  $i\geq 0$. Then there exists $\tilde{\D}_i=\{ j_i,j_i+1,\dots,j_i+ 3k\} \subseteq \D_i$ such that each position in  $\D_i$ is a possible location for the target after the first $i$ moves. \\
  $\Leftrightarrow$ For all $x \in \D_i$ there exist a sequence of adversarial moves $(d_i)_{i\in \tton}$
  such that $d_i=x$ matches $f_{\T_i}(x)$.

We define $l=\min(\T_{i+1}\cap \tilde{\D}_i)-1$,
  $r=N-\max( \T_{i+1}\cap \tilde{\D}_i)$, and the following regions (see also Figure~\ref{case}:

  \begin{caseof}
  \tcase{$|\tilde{\D}_i \cap \T_{i+1}|\geq k+1 \text{ and } l,r \geq k$\label{c:1}}{}
  \tcase{$|\tilde{\D}_i \cap \T_{i+1}|\geq k+2 \text{ and } l<k$\label{c:2}}{}
  \tcase{$|\tilde{\D}_i \cap \T_{i+1}|\geq k+2 \text{ and } r<k$\label{c:3}}{}
  \tcase{$|\tilde{\D}_i \cap \T_{i+1}|> 2k$\label{c:4}}{}
  \tcase{$\text{otherwise}$ $(\neg(\Cref{c:1}\lor\Cref{c:2}\lor\Cref{c:3}\lor\Cref{c:4}))$\label{c:5}}{}
  \end{caseof}

\Cref{c:1} $d_{i+1} \colon y_{i+1}=1 \rightarrow$
\[ \D_{i+1} = \Gamma_k(\D_i\cap \T_{i+1})\supseteq \Gamma_k(\tilde{\D}_i \cap \T_{i+1})\supseteq \{j_i-k,\dots,j_i,\dots,j_i+ 1+2k\}\]

$j_{i+1} = j_{i}- k$

\Cref{c:2} $d_{i+1} \colon y_{i+1}=1 \rightarrow$
\[ \D_{i+1} = \Gamma_k(\D_i\cap \T_{i+1})\supseteq \Gamma_k(\tilde{\D}_i \cap \T_{i+1})\supseteq
  \{1, 2, \dots, 3k + 1\} \]

\Cref{c:3} $d_{i+1} \colon y_{i+1}=1 \rightarrow$
\[ \D_{i+1} = \Gamma_k(\D_i\cap \T_{i+1})\supseteq \Gamma_k(\tilde{\D}_i \cap \T_{i+1})\supseteq
  \{ N-3k, N-3k+1, \dots, N \} \]

\Cref{c:4} $d_{i+1} \colon y_{i+1}=1 \rightarrow$
\[ \D_{i+1} = \Gamma_k(\D_i\cap \T_{i+1})\supseteq \Gamma_k(\tilde{\D}_i \cap \T_{i+1}) \]
$\mathrm{if} j_i \geq k + 1\colon \{ j_i-k, j_i-k+1, \dots, j_i + 2k\}$
$\mathrm{otherwise}  \{ j_i, j_i+1, \dots, j_i + 3k\}$

\Cref{c:5}
$d_{i+1} \colon y_{i+1}=0 \rightarrow |\D_i \cap \T_{i+1}| < k+1$\\
$ |\tilde{\D_i}=3k+1| \rightarrow |\tilde{\D_i} \setminus \T_{i+1}| \geq 3k+1-k = 2k+1$\\
if $l>k\colon$ else $r>k$\\
$3k > |\tilde{\D_i} \setminus \T_{i+1}| \geq 2k$\\
case 1: $|\tilde{\D_i} \setminus \T_{i+1}| \geq 3k$ trivial, since $4k+1-3k=k+1$(max gap)
and $k+1$ can be filled up with neighborhood relation.

\begin{enumerate}
\item In this case let $\T_{i+1}\cap \tilde{\D}_i=\{j_i,\dots, j_i+ k+1\}$, then
    \[
      \D_{i+1}=\Gamma(\T_{i+1}\cap \D_i)\supseteq \Gamma(\T_{i+1}\cap \tilde{\D}_i)\supseteq \{j_i-k,\dots,j_i,\dots,j_i+ 1+2k\} 
    \]
    and the claim holds for $i+ 1$ by choosing $j_{i+1}=j_i-k$. In this case there are enough vertices on the left and right so that the claim is fulfilled.
    % Analogously it works for the other cases in (1.).
     \item In this case let $\T_{i+1}\cap \tilde{\D}_i=\{j_i,\dots, j_i+ 2k-l\}$, then \[
     \{1,2,\dots,3k+1\}\subseteq \D_{i+1}
    \]
    and the claim holds for $i+ 1$. In this case there are enough vertices on the right so that the claim is fulfilled.
    \item In this case let $\T_{i+1}\cap \tilde{\D}_i=\{j_i,\dots, j_i+ 2k-r\}$, then \[
     \{N-3k,N-3k+1,\dots,N\}\subseteq \D_{i+1}
   \]
    and the claim holds for $i+ 1$. In this case there are enough vertices on the left so that the claim is fulfilled.
    \item In this case let $\T_{i+1}\cap \tilde{\D}_i=\{j_i,\dots, j_i+ 2k\}$ if $j_i\geq k+1$, then \[
     \{j_i-k,j_i-k+1,\dots,j_i+2k\}\subseteq \D_{i+1}
    \]
    otherwise
    \[
     \{j_i,j_i+1,\dots,j_i+3k\}\subseteq \D_{i+1}
   \]
    and the claim holds for $i+ 1$. In this case there are enough vertices on the right or left so that the claim is fulfilled.
%It is not hard to check
% that the following holds true
%  \[
%     \D_{i+1}=\Gamma(\T_{i+1}\cap \D_i)\supseteq \Gamma(\T_{i+1}\cap \tilde{\D}_i)\supseteq 
%     \tilde{\D}_i.\]
%     Hence, the claim holds for $i + 1$ by choosing $j_{i+1} = j_i$.
    \item[5.] If neither (1.), (2.), (3.), or (4.) hold, then 
    $|\tilde{\D}_i\setminus \T_{i+1}|\geq 2k$.
    
    Since $N\geq 4k+1$, we also have that 

if $3k> |\tilde{\D}_i\setminus \T_{i+1}|\geq 2k$, then 
$$\min\{  x\in \tilde{\D}_i\setminus \T_{i+1}\}>k ,$$
$$\max(\{\tilde{\D}_i\setminus \T_{i+1}\})< N-k+1.$$
In this case, the complementary set is large enough so that if the answer is "0", the claim is fulfilled.
\end{enumerate}

\begin{align*}
   \text{$y_{i+1}$} =
    \begin{cases}
    \text{$1$} & if\ 1.\ |\T_{i+1} \cap \tilde{\D}_i|\geq k+1 \text{ and } l,r \geq k
    \\
    \text{$1$} &  if\ 2.\ |\T_{i+1} \cap \tilde{\D}_i|+l \geq  2k+1 \text{ and }  l<k
    \\
    \text{$1$} & if\ 3.\ |\T_{i+1}\cap \tilde{\D}_i|+r\geq 2k+1 \text{ and } r<k
    \\
    \text{$1$} & if\ 4.\ |\T_{i+1} \cap \tilde{\D}_i|> 2k
    \\
    \text{$0$} & if\ 5.\  \text{otherwise}.
    \end{cases}
\end{align*}

\begin{figure}[ht]
  \centering
  \begin{tikzpicture}[scale=0.7, transform shape]
  \node[draw, align=left] (DcTleq2k) {If not \Cref{c:4}, then $|\tilde{\D}_i \cap \T_{i+1}| \leq 2k$ \\$|\tilde{\D}_i \cap \T_{i+1}|<k+1?$};
  \node[draw, align=left, below left= 1 and 1 of DcTleq2k] (DnTgt2k+1) {$|\tilde{\D}_i \setminus \T_{i+1}| > 2k+1$ \\$|\tilde{\D}_i \setminus \T_{i+1}|+k > 3k+1$};
  \node[draw, align=left, below right= 1 and -2 of DcTleq2k] (DcTeqk+1) {$|\tilde{\D}_i \cap \T_{i+1}| = k+1?$};
  \node[draw, align=left, below left= 1.5 and 2 of DcTeqk+1] (LRgeqk) {$l \geq k \land r \geq k?$};
  \node[draw, align=left, below right = 1 and -3 of DcTeqk+1] (DcTgeqk+2) {$|\tilde{\D}_i \cap \T_{i+1}| \geq k+2$ \\$l>k?$};
  \node[draw, align=left, below left= 1 and 0 of LRgeqk] (c1) {\Cref{c:1}$|\tilde{\D}_i \cap \T_{i+1}|=3k+1$\\$(k+1+2k)$};
  \node[draw, align=left, below right = 2.3 and -1 of LRgeqk] (Lltk) {$l < k?$};
  \node[draw, align=left, below left= 1.5 and -1 of Lltk] (c5) {$r = N - |\tilde{\D}_i \cap \T_{i+1}|, N \geq 4k+1$
    \\$ r \geq 3k+1 \curvearrowright \Cref{c:5}:$
    \\$ |\tilde{\D}_i \setminus \T_{i+1}| = 3k+1 -(k+1)= 2k$
    \\$ (\tilde{\D}_i \cap \T_{i+1})$ starts $<k$ from left
    \\ and ranges over $k+1$
    \\ $|\tilde{\D}_i| = 3k+1 \ |\tilde{\D}_i| - |\tilde{\D}_i \cap \T_{i+1}| = 2k$
    \\ $2k>k$ pos. counted from left for $|\tilde{\D}_i \setminus \T_{i+1}|$};
  \node[draw, align=left, below right= 2 and 1 of Lltk] (Rltk) {$r < k$
    \\ $ l \geq N - |\tilde{\D}_i \cap \T_{i+1}|,$
    \\$ N \geq 4k+1$
    \\$ r \geq 3k+1 \curvearrowright \Cref{c:5}$
    % \\$ |\tilde{\D}_i \setminus \T_{i+1}| = 3k+1 -(k+1)= 2k$
    % \\$ (\tilde{\D}_i \cap \T_{i+1})$ starts $<k$ from right
    % \\ and ranges over $k+1$
    % \\ $|\tilde{\D}_i| = 3k+1 \ |\tilde{\D}_i| - |\tilde{\D}_i \cap \T_{i+1}| = 2k$
    % \\ $2k>k$ pos. counted from left for $|\tilde{\D}_i \setminus \T_{i+1}|\curvearrowright$
  };
  \node[draw, align=left, below left= 2 and -0.5 of Rltk] (5concl) {$\curvearrowright$ enough space for $\Gamma_k$-expansion
    \\ on both sides};
  \node[draw, align=left, below= 1 of DcTgeqk+2] (Lgtk) {$l > k?$};
  \node[draw, align=left, below left= 0.5 and 0.2 of Lgtk] (c3) {$r<k$\\ \Cref{c:3}};
  \node[draw, align=left, below right= 0.5 and 0.2 of Lgtk] (c2) {$r>k$\\ \Cref{c:2}};
  \draw[->] (DcTleq2k)--(DnTgt2k+1) node[midway, below]{yes};
  \draw[->] (DcTleq2k)--(DcTeqk+1) node[midway, below]{no};
  \draw[->] (DcTeqk+1)--(LRgeqk) node[midway, below]{yes};
  \draw[->] (DcTeqk+1)--(DcTgeqk+2) node[midway, right]{no};
  \draw[->] (LRgeqk)--(c1) node[midway, right]{yes};
  \draw[->] (LRgeqk)--(Lltk) node[midway, right]{no};
  \draw[->] (Lltk)--(c5) node[midway, right]{yes};
  \draw[->] (Lltk)--(Rltk) node[midway, right]{no};
  \draw[->] (Lgtk)--(c3) node[midway, above left]{yes};
  \draw[->] (Lgtk)--(c2) node[midway, above right]{no};
  \draw[->] (c5)--(5concl) node[midway, below left]{yes};
  \draw[->] (Rltk)--(5concl) node[midway, below]{no};
  \draw[->] (DcTgeqk+2)--(Lgtk) node[midway, right]{};
\end{tikzpicture}
  \caption{Path through different cases of Proposition~\ref{pcase}}
  \label{case}
\end{figure}
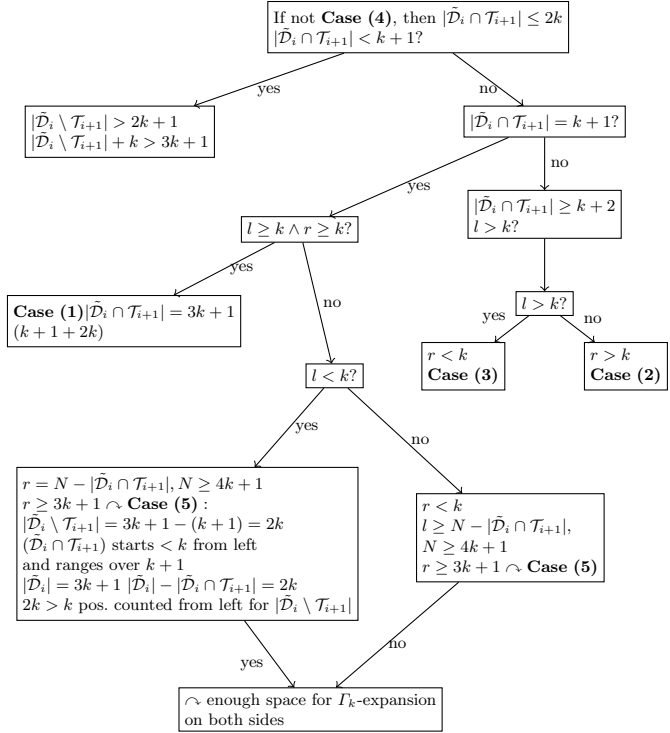

\end{proof}

% ---Stop editing

%\msa{what is the range of $k$?}
%\msa{the notation $P_N$ is new ! its better to stick to the previously defined, i.e., $C_N$ otherwise define $P_N$}
\begin{proof}
% We fix any strategy, and let $n$ denote its number of tests. \ogr{Let $N \geq 4k+1$.}
% We can define an adversarial
% strategy for the target, i.e., a sequence of test results that guarantees the following claim, and which implies the desired result.
% \begin{Claim}
% Let $k\geq 1$. For each $i\geq 0$ there  exists $j_i\geq 1$, such that the  set  of  possible  locations  for  the  target  always  contains $3k+1$ consecutive locations, i.e.,
% \[
% \D_i \supseteq \{j_i,j_i+1,j_i+2,\dots ,j_i+3k\}.
% \]
% \end{Claim}
%\msa{range of $k$? is it $k\geq 1$? then please add it to the claim statement.}

%

%

%\cde{Missing items}
\qed
\end{proof}
%%%
\begin{Theorem}
For $k\geq 1$, the minimum $s$ such that there exists a $(G,s)$-successful strategy is given by
\begin{align*}
    s^* = \begin{cases}
    N &  \text{ for } N \leq 2k+1
    \\
    \lceil\frac N2\rceil+k & \text{ for } 2k+1 < N < 4k+1
    \\
    3k+1 & \text{ for } N \geq 4k+1.
    \end{cases}
\end{align*}
\end{Theorem}
%%%
\begin{proof}
The first case is obvious, as after each test we can choose the result when $|\D_i|=N$  .

The last case follows from Proposition 2 as well as the following algorithm (called "shifting algorithm"). $$\T_1=\{1,2,\dots \lceil\frac N2\rceil\}.$$

Without loss of generality, we assume that the answer to the first test of the shifting algorithm is 1.
If $N<4k+3$, we can stop and give the final answer with $s=3k+1$, otherwise
for $1<i\leq \lceil\frac N2\rceil\ -2k$, we choose
$$\T_i= \{\lceil\frac N2\rceil\ + (2-i),\lceil\frac N2\rceil\ + (2-i)+1,\dots \lceil\frac N2\rceil\ + (2-i)+k\}.$$ 
We have $|\T_i|=k+1$. As soon as we get the test result $1$, we can stop the algorithm
and $s=3k+1$. If all test results are $0$, $|\D_{\lceil\frac N2\rceil\ -2k}|=\{1,2, \dots 3k+1\}$. Thus we always have $s=3k+1.$

The case for $2k+1 < N < 4k+1$ we can get from a 1-test strategy (when $\T_1=\{1,2,\dots \lceil\frac N2\rceil\}$). For such $N$ we have $\lceil\frac N2\rceil+k < 3k+1$. We can't improve the result with more tests. Let $D_1=\{1,2,\dots \lceil\frac N2\rceil+k\}$. For $T_2=\{1,2,..z\}$ or $T_2=\{\lceil\frac N2\rceil+k, \lceil\frac N2\rceil+k-1,\dots \lceil\frac N2\rceil+k-z+1\}$, we have the following variations.

If $z\leq k$, then the answer is $0$ and $|\D_2|\geq |\D_1|$. Otherwise, we have only 2 variations.

If $z\geq \lceil\frac N2\rceil$, then the test result is $1$ and $|\D_2|\geq |\D_1|$. Otherwise, the result is $0$ and $$|\D_2|\geq \min\{\lceil\frac N2\rceil+k, 3k\} \geq |\D_1|.$$

Finally, let $T_2$ do not include the vertices $1$ and $\lceil\frac N2\rceil$. Then if $|\T_2 \cap \D_1|\leq 2k$ then the test result is $0$ and $|\D_2|\geq \min\{|\D_1|+k,N\}$. Otherwise, the test result is $1$ and $|\D_2|\geq \min\{3k,N\}$.

Obviously all variations without connecting sets are worse.
\qed

\end{proof}

%\ale{Shift correct?}

For $k>1$ we have to consider two cases.
First we consider the case if $3k+1\leq s\leq 4k$. In this case, a shifting algorithm will be optimal.
%%%
\begin{Theorem}
Let $n\geq 1$ and $1\leq l\leq k$. Then the maximum $N$, such that there exists a $(P_N,s)$-successful strategy with $n$ tests is given by
\[ N_p(n,3k+l)\geq 2nl+4k. \]
\end{Theorem}
%%%
\begin{proof}
The idea of the proof is to take a "sliding window" of size $l+k$ after the first test and move it $l$ steps to the left or right 
depending on the test result of the first test. We let the first test be defined as $\T_1=\{1,2,\dots \lceil\frac N2\rceil\}$. W.l.o.g. we can assume that the test result is $1$ or the test result is $0$ and $\lceil\frac  N2\rceil \leq 2k+l$ and then we stop. Otherwise, \[ \T_2=\{\lceil\frac N2\rceil-(l-1),\lceil\frac N2\rceil -(l-2), \dots \lceil\frac N2\rceil+k\}.\]
Now if the 2nd test result is $1$ or the test result is $0$ and 
$\lceil\frac N2\rceil-l\leq 2k+l$, again we stop.
Otherwise, \[
\T_3=\{\lceil\frac N2\rceil-2l+1,\lceil\frac N2\rceil-2l+2, \dots \lceil\frac N2\rceil+k-l\}.
\]
Now if the 3rd test result is $1$ or the test result is $0$ and $\frac N2-2l\leq 2k+l$, again we stop. Therefore, we can repeat the procedure, and in general we have \[
\T_i=\{\lceil\frac N2\rceil-(i-1)l+1,\lceil\frac N2\rceil-(i-1)l+2, \dots \lceil\frac N2\rceil+k-(i-2)l\}
\]
for $i>1$. If the $i$-th test result is $1$ or the test result is $0$ and 
$\lceil\frac N2\rceil-(i-1)l\leq 2k+l$, we stop.
As a result, after $n$ steps we will successfully terminate if $N=2nl+4k$.
\end{proof}
%%%
Now we consider the case
$s\geq 4k+1$ and characterize the optimal strategies for the path graph. We start with the upper bound.
% \begin{Theorem}
% For $n\geq 1$ it holds
% \[
% N_p(n,4k+1)\geq 2^n+k(2n+4).
% \]
% \end{Theorem}
% \textcolor{red}{next to be corrected}
% Perhaps give special case for old model.
% Argue for new model and then prove for new model....
% \[ N_p(n,3k+1)\leq 2n+2k. \]
%%%
\begin{Theorem}\label{T4}
Let $k\geq 1$ for $n\geq 0$ and $s\geq 4k$. Then for the maximum $N$, such that there exists a $(P_N,s)$-successful strategy with $n$ tests, it holds that
\[ N_p(n,s)\leq (s-4k)2^{n}+k(2n+4). \]
\end{Theorem}
%%%
\begin{proof}
Fix $s\geq 4k$. Skipping trivialities, let us assume that $n > 0$.
Let $N=2^n(s-4k)+2kn+4k+1$ and $S$ be any testing strategy with $n$ tests.
We shall describe an adversary
strategy (i.e., a sequence of test results) proving that $S$ is not a successful strategy.
Let $A_0 =\{nk+1,nk+2,\dots, nk+2^n(s-4k)+4k+1\}$. 
Notice that $N\setminus A_0$ does not include less than $nk$ positions on each side of $A_0$.
Let $\T_1$ be the first test in the strategy $S$. 
%\cde{copy paste error: check brakets}
%\msa{I think there is no error so you can remove this remark}
%
\begin{align*}
    \text{$y_1$} =
    \begin{cases}
    \text{$0$} & \text{ if } |\Gamma(\T_1)\cap A_0|< |\Gamma(A_0\setminus \T_1)|
    \\
    \text{$1$} & \text{ otherwise}
    \end{cases}
\end{align*}
%
% If $|\Gamma(\T_1)\cap A_0|< |\Gamma(A_0\setminus \T_1)|$ then
% the test result is $0$, otherwise the test result is 'Yes'. 
Let $A_1 = \max\{(\Gamma(\T_1)\cap A_0),\Gamma(A_0\setminus \T_1)\}$.
Since on both sides of $A_0$ there are $nk$ positions (i.e., in particular, $A_0$ does not include $1$ nor $N$), it follows that 
\[
|A_1|\geq \left\lceil \frac {|A_0|}2 \right\rceil + 2k \geq 2^{n-1}(s-4k)+4k+1.
\]
Notice also that between the leftmost (rightmost) element of $A_1$ and the leftmost
(rightmost) element of $N$ there are at least $(n-1)k$ positions.
In addition, we have $\D_1\supseteq A_1$. Hence, after the first test, there is a possible
outcome for which the size of the set of candidate positions for the target is at
least $2^{n-1}(s-4k)+4k+1$. So if $n$ was $1$, this would be proof that strategy $S$ is not successful.

In general, the adversary's strategy is described as follows:
\begin{align*}
    \text{$y_i$} =
    \begin{cases}
    \text{$0$} & \text{ if } |\Gamma(\T_i)\cap A_{i-1})|< |\Gamma(A_{i-1}\setminus \T_i)|
    \\
    \text{$1$} & \text{ otherwise},
    \end{cases}
\end{align*}
for $i = 1,2,\dots, n$.
% for each $i =1,2,\dots, n$
% if $|\Gamma(\T_i)\cap A_{i-1})|< |\Gamma(A_{i-1}\setminus \T_i)|$ then the test result will be '$0$', otherwise the test result
% will be '$1$'.

 Let $A_i = \max\{(\Gamma(\T_i)\cap A_{i-1}),\Gamma(A_{i-1}\setminus \T_i)\}$.
Then, as before, we have 
\[
|A_i|\geq \left\lceil \frac {|A_{i-1}|}2 \right\rceil + 2k \geq 2^{n-i}(s-4k)+4k+1
\]
and between
the leftmost (rightmost) element of $A_i$ and the leftmost (rightmost) element of $N$
there are at least $k(n-i)$ positions. This is important for guaranteeing the invariant
$|A_i|\geq \left\lceil \frac {|A_{i-1}|}2 \right\rceil + 2k$
as long as $i\leq n$. 
Finally, we have $\D_i\supseteq A_i$.
Therefore, after $n$ tests, we have $|\D_n|\geq s+1$, which means that the strategy $S$ is not successful.
\qed
\end{proof}
We will now show a tied lower bound.
%%%
\begin{Theorem}\label{T5}
For $n\geq 0$ and $s\geq 4k$, the maximum $N$, such that there exists a $(P_N,s)$-successful strategy with $n$ tests, is given by 
%\msa{'we have' can be removed and instead write: the maximum $N$, such that there exists a $(P_N,s)$–successful strategy with $n$ tests is given by (the same phrase can be written for Corollary.1, 2 and 3)}

\[ N_p(n,s)\geq (s-4k)2^{n}+k(2n+4). \]
\end{Theorem}
%%%
\begin{proof}
We consider two variants of the problem and then reduce the original problem
to them. The first variant (later referred to as variant O, for open) arises when
we consider the search space open on both sides. More precisely, we assume that
$\D_0 =\{a + 1,\cdots, a + x\}$ for some $a\in\ZZ$ and $x\in\NN$ and the target can move on
any position in $\ZZ$, i.e., there is no boundary at $a$ or $a + x$, in the sense that from
position $a + x$ the target can move to position $a + x + i$ with $1\leq i\leq k$ too, 
%\msa{perhaps remove the word 'too' and instead add word 'still' before the verb 'move'}
and from position
$a + 1$ it can also move to position $a-i$ with $0\leq i\leq k-1$. This is different from the problem we fixed at
the beginning, since when the target is in $1$ (resp. in $N$), if it moves, it can only
move to $2$ (resp. $N-1$).
\setcounter{Claim}{0}
\begin{Claim}
Let $N^O(n,s)$ denote the largest value of $x$ such that there is a strategy with $n$ tests which allows us to locate the target within the accuracy
$s$ on an infinite path, assuming that the target starts in some initial sub-path of $x$ vertices $\{a + 1, \cdots, a + x\}$. %\msa{either shift to Sec.~2 (Model and Definitions) or bring it in the Claim.~1 statement directly.}
\[
N^O(s,n)\geq 2^n(s-4k)+4k
\]
\end{Claim}
The base case $n = 0$ is trivially true. For the induction step, let $n \geq 1$ and
$\D_0 = \{a + 1,\dots,a + 2^n(s-4k) + 2k\}$. Using the first test 
$$\T_1 = \{a + 1,\dots,a + 2^{n-1}(s-4k) + 2k\},$$ we have that either 
$$\D_1 = \{a + 1-k,\dots,a + 2^{n-1}(s-4k) + 3k\}$$ or
$$\D_1 = \{a + 2^{n-1}(s-4k) + 1+k, \dots,a + 2^{n}(s-4k) + 5k\}.$$
In both cases we
have that $|\D_1| = 2^n(s-4k)+4k$. Hence, by the induction hypothesis,
$n-1$ additional tests are sufficient for a successful strategy starting from $\D_1$. Thus
$n$ tests are sufficient for a successful strategy starting from $\D_0$; i.e., we have shown
$N^O(s,n)\geq 2^n(s-4k)+4k$, concluding the proof of Claim~1.

As a second variant of the problem,
we consider the case where the search space is half-open (later referred to as variant H). We assume that
$\D_0 = \{1,\dots,x\}$ for some $x\in\NN$ and the target can move on any position in $\NN$, i.e.,
it can never move to a position to the left of 1 but it can move to a position to the
right of $x$, meaning, positions $x+1, x+2,\dots$ might become possible candidates later on.
By induction we can prove the following.
\begin{Claim}
Let us denote by $N^H(n, s)$ the largest value of $x$ such that there is a strategy with $n$ tests which allows us to locate the target within accuracy $s$ on a semi-infinite path $\{1,2,\dots\}$, assuming that the target starts on the initial sub-path of $x$ vertices $\{1,\dots,x\}$.  
%\msa{all the last paragraph can be written directly in Claim.~2 statement}
\[
N^H(n,s)\geq 2^n(s-4k)+(n+4)k
\]
\end{Claim}
The base case $n = 0$ is trivially true. For the induction step, let $n\geq 1$ and
$\D_0 = \{1,\dots,2^n(s-4k)+k(n+4)\}$. Using the first test
\begin{align*}
     \T_1 = \{ 1,\dots,2^{n-1}(s-4k)+k(n+2)\},
\end{align*}
it follows that either
\begin{align*}
    \D_1 = \{ 1,\dots,2^{n-1}(s-4k)+k(n+3)\}
\end{align*}
or
\begin{align*}
    \D_1 = \{ 2^{n-1}(s-4k)+k(n+1)+1, \dots , 2^n(s-4k)+k(n+5)\}.
\end{align*}
In the first case we have that
$|\D_1|\leq 2^{n-1}(s-4k)+k((n-1)+4)\leq N^H(n-1,s)$, hence by induction $n-1$ additional
tests are sufficient for a successful strategy starting from $\D_1$. In the second case, we
have $|\D_1|=2^{n-1}(s-4k)+4k\leq N^O(n-1,s)$, and even allowing the new search space
to be open, we can finish the search with a strategy of size $n-1$, by Claim~1.
In both cases, $n$ tests are sufficient for a successful strategy starting from $\D_0$, hence the inductive step is established. This completes the proof of Claim~2.

Now we address the proof for the lower bound, i.e., $$N_p(n,s)\geq (s-4k)2^{n}+k(2n+4).$$
The case $n = 0$ is trivially true. For $n\geq 1$, let 
$\D_0 = \{1,\dots,2^n(s-4k)+k(2n+4)\}$ and
define 
$\T_1 = \{ 1,\dots,2^{n-1}(s-4k)+k((n-1)+3)\}$.
We have that either 
$\D_1 = \{ 1,\dots,2^{n-1}(s-4k)+k((n-1)+4)\}$
or 
$\D_1 = \{ 2^{n-1}(s-4k)+k((n-1)+2)+1, \dots , 2^n(s-4k)+k(2n+4)\}$.
In both cases we have 
$|\D_1|\leq N^O(n-1,s)$ and we can finish in $n-1$ tests
by Claim~2, using a strategy for the half-open variant defined above. Notice that
in both cases, the resulting set of candidate positions for the target can extend
only in one direction, like in an instance of the half-open variant. 
Hence $N_p(n,s)\geq (s-4k)2^{n}+k(2n+4)$, as desired.
\qed
\end{proof}
Using Theorem~\ref{T4} and Theorem~\ref{T5}, we obtain the following.
%We get the following. %\msa{perhaps, this sentence is not necessary and can be removed or at least use a more illustrating explanation if you want to connect the bridge to next corollary}
\begin{Corollary}
For $n\geq 0$ and $s\geq 4k$, the maximum $N$, such that there exists a $(P_N,s)$-successful strategy with $n$ tests, is given by 
\[ N_p(n,s)= (s-4k)2^{n}+k(2n+4). \]
\end{Corollary}
%\msa{Maybe adding one or two simple figure that illustrate a graph or tree related to the search algorithms looks good}

\section{Non-Adaptive Two-Sided Search on a Path Graph}

In this section we give an optimal non-adaptive strategy
for the search in a path graph. In this section we assume that
$k=1$ which means the object can only make movements of size $1$.
Let $P_N=(\N,\E)$ be a path graph on $N$ vertices, where the set of edges is given by
\begin{gather*}
    \E=\{(i,i+1)\}_{i\in\N\backslash N}~\bigcup~ \{(i,i)\}_{1 \leq i \leq N}.
\end{gather*}
It was shown in \cite{ACDL13} that for an adaptive strategy we have $s^*(P_N)=4$, and thus the minimum $n$ such that there exists a successful strategy with $n$ tests is given by
%%%
\begin{equation}
\label{ADK1}
n_A(P_N,4)= \left\lceil\frac N2\right\rceil -2. 
\end{equation}
%%%
We will show now that there exists a non-adaptive strategy, which needs the same number of tests as an optimal adaptive strategy. Therefore this strategy is also optimal. We first give an example for $N=16$. We present a strategy that needs $\left\lceil\frac N2\right\rceil -2=6$ tests and therefore is optimal.
\begin{Example}
Consider a non-adaptive strategy $S(n,N)=A_{n,N}=\{a_{ij}\}$ where $n=6$ and $N=16$ as follows
\setcounter{MaxMatrixCols}{20}
%%%
\[
\begin{pmatrix}
0 & 0 & 0 & 0 & 0 & 0 & 0 & 0 & 1 & 1 & 1 & 1 & 1 & 1 & 1 & 1\\ 
0 & 0 & 0 & 0 & 0 & 0 & 0 & 1 & 1 & 0 & 0 & 0 & 0 & 0 & 0 & 0\\ 
0 & 0 & 0 & 0 & 0 & 0 & 1 & 1 & 0 & 0 & 1 & 1 & 1 & 1 & 1 & 1\\ 
0 & 0 & 0 & 0 & 0 & 1 & 1 & 0 & 0 & 1 & 1 & 0 & 0 & 0 & 0 & 0\\ 
0 & 0 & 0 & 0 & 1 & 1 & 0 & 0 & 1 & 1 & 0 & 0 & 1 & 1 & 1 & 1\\ 
0 & 0 & 0 & 1 & 1 & 0 & 0 & 1 & 1 & 0 & 0 & 1 & 1 & 0 & 0 & 0
\end{pmatrix}.
\]
%%%
  Let us assume the first test result is $y_1 = 0$. We will then know
  the prior position of the target, once we receive the first time a
  test result $ y_i = 1, \ i\in\{2, 3, 4, 5, 6\}$. If $y_i = 0 $ in
  all tests, then $d_{6} \in \{ 1, 2, 3 \}$. Taking the expansion into
  the neighborhood at the right side into account, we have that
  $|\D_6| = |\{ 1, 2, 3, 4\}| = 4 .$ Otherwise if the first test
  result is $y_1 = 0$, then we know the prior position once we have
  the first time a test result $ y_i = 1, \ i\in \{ 2, 4, 6 \} $ or
  $ y_i = 0, \ i\in \{ 3, 5\}.$

%The reader who is familiar with the adaptive strategy of \cite{ACDL13} will notice the similarities. 
%\msa{alternative: The technique for understanding and analysis of the matrix are similar to that of the given for an adaptive strategy in [5]} 
%Roughly speaking,
%\msa{perhaps 'almost' or 'approximately' can be used instead of 'roughly speaking'}, 
%the tests \ogr{for a right and a left sub-path} are asked in parallel.
\end{Example}

%\sout{Next we define for every $n$ and $N$ a non-adaptive 
%strategy (matrix) and generalize our example.}
Now we will
show that the strategy is successful for the general case. 
\begin{Theorem}\label{nonadaptive}
The minimum $n$ such that there exists a successful strategy with $n$ tests is given by
\[ n(P_N,4)= \left\lceil\frac N2\right\rceil -2. \]
\end{Theorem}
%%%

\begin{proof}
We already 
know from equation~\eqref{ADK1} that for the adaptive case we
need no fewer than $\left\lceil\frac N2\right\rceil -2$ tests. Therefore, this also holds 
for the non-adaptive case.

We also do not need more than this amount of tests, since we can construct a successful strategy for each $N$.

This strategy bisects the number of nodes in the path into two, such that the graph is divided into two, almost of equal size in the number of nodes and sub paths. The nodes of one path belong to the test set, the others do not. This ensures that the test provides the knowledge, whether the target was on the first subpath or on the other one.

In order to avoid an expanding area of uncertainty at the connecting end of the two paths, which results from the possible movement of the target, we delimit the ends by a region of
neighbor nodes, which are alternating by their membership in the current test set.

By expanding the center region, we increase the accuracy of each test, since the number of nodes at the two paths is declining. Once the target gets from the end of one sub path into the region of alternating test membership, the prior position is determined and the target can only be located at the position of the bordering two nodes, expanded by the possible moves to each side. This results in an accuracy of s=4, which proves that the strategy is successful.

The strategy is formalized by the following.

%\begin{Definition}\label{NAS}
  Let $n,N\in\NN$ be given. For $1 \leq i \leq  n$ and $1 \leq j \leq N$, we define
  the \textit{expanding accuracy strategy} $S(n,N)=A_{n,N}=\{a_{ij}\}$ by 

  \begin{enumerate}[label=(\roman*)]
    \item $a_{ij}=0$, for $1\leq j \leq \left\lceil\frac N2\right\rceil -i+1$
    \\
    \item $a_{ij}=
    \begin{cases}
        0 & \text{ if $i$ = even} \\
        1 & \text{ if $i$ = odd}
    \end{cases},$ for $\left\lceil\frac N2\right\rceil+i\leq j \leq N$.
  \item $a_{i,\left\lceil\frac N2\right\rceil +p-i+1} = g_p$, otherwise,\\

    where
    \begin{align*}
g_p = \begin{cases}
    1 & p \equiv 1\  (mod\ 4)\text{ or } p \equiv 2\ (mod\ 4)
    \\
    0 & \text{otherwise}
    \end{cases}
    \text{    , for $p\geq 1$. 
    %\msa{Is the upper-bound of $i$, $1$ or $n$?}
    }
\end{align*}
  \end{enumerate}
  
%\end{Definition}

\end{proof}

\begin{Remark}
The result of Theorem 6 is surprising in the sense that it shows that our optimal non-adaptive strategy requires as many tests as the optimal adaptive strategy derived in \cite{ACDL13}. For practical applications, this means that you don't have to wait for the test result but can run the next test immediately. This can be a huge time saver. 
\end{Remark}

\begin{Proposition}
\label{Proposition3}
For $N> 6k$ there does not exist a $(P_N,s)$-successful non-adaptive strategy with $s<4k$.
\end{Proposition}
\begin{proof}
Obviously, if the strategy contains only $1$ test, we can't get an answer with $s<4k$. Let our non-adaptive strategy consist of $t\geq 2$ tests. We consider values in $2k+1$ middle vertices (from $\left\lceil\frac N2\right\rceil-k$ to $\left\lceil\frac N2\right\rceil+k$) in test $t$. If all of them are the same, we can use following counter-strategy.

%\sout{Results of tests $1 \dots t-1$ is $f_{\T_i}(\left\lceil\frac %N2\right\rceil)$, so before test $t$, the target might be anywhere between %$\left\lceil\frac N2\right\rceil-k$ and $\left\lceil\frac N2\right\rceil+k$. The %result of test $t$ is also $f_{\T_i}(\left\lceil\frac N2\right\rceil)$, so after the movement, our target will be somewhere between $\left\lceil\frac N2\right\rceil-2k$ and $\left\lceil\frac N2\right\rceil+2k$, so in this case $s\geq 4k+1$.

%Now we consider the case of different values in $2k+1$ vertices from $\left\lceil\frac N2\right\rceil-k$ to $\left\lceil\frac N2\right\rceil+k$ in test $t$. Let us have two neighborhood vertices $x^*$ and $x^*+1$ with different values. We can use the following counter-strategy for $x^*\leq \left\lceil\frac N2\right\rceil$ (and there is a similar one for $x^*>\left\lceil\frac N2\right\rceil$).

%In test $i$ the result is}

We restrict the target position to reside in all such cases in the test set, when the center $\left\lceil\frac{N}{2}\right\rceil$
node is within the test set.
Without that restriction, the target preserves a higher flexibility, which means that it can hide in larger elements of the bipartition defined by the test set.

We know for the first test in the induction begin, that we can not achieve a higher accuracy than $4k$.
We assume that we still did not achieve a higher accuracy in step $t-1$.

However, if the the test $\T_t$ excludes or includes $2k+1$ middle nodes from $\left\lceil\frac N2\right\rceil-k$ to $\left\lceil\frac N2\right\rceil+k$ ) as a whole, and the target resides in these nodes, then we have an accuracy of $s*=4k+1>4k$
after test $\T_t.$

If the nodes from $\left\lceil\frac N2\right\rceil-k$ to $\left\lceil\frac N2\right\rceil+k$ are only partially included in test $\T_{t}$, then there exists one pair of neighbor positions, $x^{*}$ and $x^{*}+1$, which have different values in the matrix.
Let us assume w.l.o.g., that $x^{*}\leq \left\lceil\frac N2\right\rceil.$
At position $x^{*}+2k$ at test round $t$, there is a position which is either within or outside of test $\T_t$. Now, if $x^{*}$ has the same test membership as $x^{*}+2k$,
we consider $x^{*}$ as a second possible target position, aside from $x^{*}+2k$. Otherwise, we consider $x^{*}+1$ as the second possible target position.

\begin{align*}
    y_i = 
    \begin{cases}
        f_{\T_i}(x^*+k) & \text{for } i=1,\dots,t-1  \\
        f_{\T_i}(x^*+2k) & \text{for } i=t .
    \end{cases}
\end{align*}
%
% Answers for test $1,\dots,t-1$ is the value of point $x_1+k$. And for test $t$ answer is value of point $x_1+2k$.

The target positions are possible because we restricted the target in the steps $t-1$ and kept the positions that were agreed upon the test membership with node $x^{*}+k$.
Both, the first and the second possible target positions are reachable from $x^{*}+k$. Taking the neighborhood expansion into account, we have a closed area of possible target positions
after test $t$, which extends from $x^{*}$, (or $x^{*}+1$), via $x^{*}+k$ to $x^{*}+2k$. This interval contains $4k + 1$ positions, if $x^{*}$ and $x^{*}+2k$ share the membership, otherwise $4k$ positions.

%\sout{Meanwhile, before test $t$, the target position can be anywhere between $x^*$ and $x^*+2k$. And with the movement after test $t$, the cardinality of the possible target position will be $4k+1$ (if the value in $x^*+2k$ is the same to value in $x^*$) or $4k$ (in the other case).}
This completes the proof of Proposition~\ref{Proposition3}.\qed
\end{proof}
%%%
\begin{Theorem}\label{nami}
The minimum $s$, such that there exists a $(G,s)$-successful non-adaptive strategy is given by
\begin{align*}
    s^* = \begin{cases}
    N & \text{if } N\leq 2k
    \\
    \left\lceil\frac N2\right\rceil+k & \text{if } 2k < N \leq 6k
    \\
    4k & \text{if } N > 6k.
    \end{cases}
\end{align*}
\end{Theorem}
%%%
\begin{proof}
The first inequality is obvious.

The last one follows from Proposition~\ref{Proposition3}, and from the strategy similar to strategy 1,  where the $0\text{'s}$ and $1\text{'s}$ are substituted with $k$-successive $0\text{'s}$ and $1\text{'s}$, respectively.

The second inequality follows from a counter-strategy from the proof of Proposition~\ref{Proposition3}, but in the cases where $2k<N\leq 6k$, the object can not move to both sides for $k$. Thereby we get that  $s^*\geq \left\lceil\frac N2\right\rceil+k$. And we can get this result in the simple 1-step strategy. Observe that since $N\leq 6k$, $\left\lceil\frac N2\right\rceil+k\leq 4k$ for such $N$.
This completes the proof of Theorem~\ref{nami}.
\end{proof}

\begin{Remark}
The non-adaptive strategy for the case $N>6k$ above is optimal. Therefore, we find in this case a non-adaptive strategy for any arbitrary $k$ with the 
same number of tests as the adaptive strategy.
\end{Remark}

\section{A coding problem equivalent to the search problem}

We say that a coding problem is equivalent to a search problem if each code of the coding problem can be used simultaneously as a search strategy and vice versa. This equivalence has been known for a long time. Berlekamp, for example, used it in his PhD thesis \cite{B64} to develop error-correcting codes with feedback. In our case, too, there is an equivalent coding problem that we will briefly describe here. All the results from the previous sections also apply to this coding problem because of the equivalence.
We have a channel with a transmitter and a receiver over which we can transmit one bit per time unit without noise. Furthermore, a graph $G=(\N,\E)$ is given that corresponds to the graph from the search problem, where $\N$ denotes the positions and $\E$ denotes the possible paths. An object moves $k$ times per time unit exactly as in the previously described search problem. The starting position of the object is chosen arbitrarily. The aim of the sender now is to transmit where the object is located to the receiver with accuracy $s$ (as defined in the search problem). For this purpose, as mentioned above, he can transmit one bit per time unit without noise. For this purpose, the transmitter and receiver can use one of the search strategies developed in the previous sections. The transmitter sees the object and sends the first answer to the first test of the search strategy over the channel. Then the object moves and the transmitter sends the second answer to the second test of the search strategy. 
After $n$ tests at the latest, the receiver has received the position of the object with accuracy $s$. If the sender uses a non-adaptive search strategy, he can simply transmit all answers to all tests of the search strategy step by step in each time step. If he uses an adaptive search strategy, he must always choose the next test depending on the previous bits sent. We thus obtain an encoding strategy from our search strategy for the encoding problem described above. In the same way, each coding strategy can be used as a search strategy. It is often helpful to consider the problem as both an encoding problem and a search problem.

%%%%%%%%%%%%%%%%%%%%%%%%%%%%%%%%%%%%%%%%%%%%%%%%%%%%%%%%%%%%%%%%%%%%%%%%%%%%%%%%%%%%%%%%%%%%%%%%%%%%%%%%%%%%%%%%%%%%%%%%%%%%%%%%%%%%%%%%%%%%%%%%%%%%%%%%%
\section{Conclusions}

We have expanded the theory of combinatorial two-sided search. 

First, the case in which a searched object can move more steps between the tests is addressed. This is a claerly different from the original model, where the searched object is only permitted to move one step in between the tests. In our proposed model, the object can move up to $k$ steps between tests. This can be interpreted as the fact that the searched object is subject to a non-uniform speed spectrum and might move at different speeds.

Second, we considered the non-adaptive scenario. It has been shown that for the path graph there exists an optimal non-adaptive strategy with the same number of tests that the optimal strategy has for the same parameters.

The aim should be to analyze more complex graphs. There are examples in which an optimally non-adaptive strategy requires more tests than an optimally adaptive strategy does for more complicated graphs. In such cases, it would be interesting to consider strategies that are not completely adaptive, that is, to test in several stages.

In the course of analysis for more complicated graphs, we observe that if the model is restricted to the case where the searched object does not move after the last test, it can be advantageous in the sense of requiring significantly fewer case distinctions for the graph analysis. The graphs considered in this paper have no advantages for the analysis. 
We get the following result for the cycle graph:
For $s\geq 4k$ and $n\geq 0$, we have
\[
N^*_c(n,s)= (s-2k)2^n+2k,
\]
% \textcolor{red}{its better to bring the sentence 'For $s\geq 4k$ and $n\geq 0$' after the equation and delete the term 'we have'.}
and the following result for the path graph:
For $n\geq 0$ and $s\geq 4k$, we have
\[ N^*_p(n,s)= (s-2k)2^{n}+k(2n+2). \]
% \textcolor{red}{its better to bring the sentence 'For $n\geq 0$ and $s\geq 4k$' after the equation and delete the term 'we have'.}
% ---
\section*{Acknowledgement}
Christian Deppe was supported by the Bundesministerium f\"ur Bildung und Forschung (BMBF) through Grant 16KIS1005. Furthermore, he acknowledge the financial support by the Federal Ministry of Education and Research of Germany in the programme of “Souverän. Digital. Vernetzt.”. Joint project 6G-life, project identification number: 16KISK002. Alexey Lebedev was supported by RFBR and the National Science Foundation of Bulgaria (NSFB), project number 20-51-18002 and RFBR, project number 19-01-00364. Finally, we thank Olaf Gr\"oscho, Vladimir Lebedev and Mohammad J. Salariseddigh for their helpful and insightful comments.

%\section*{References}

\end{document}